 \documentclass[1p]{elsarticle}
\usepackage{graphicx,overpic}
\usepackage{amsmath,amssymb}
\usepackage{amsthm,mathtools,mathrsfs}
\journal{}
\makeatletter\def\ps@pprintTitle{\let\@oddhead\@empty\let\@evenhead\@empty\def\@oddfoot{}\def\@evenfoot{}}\makeatother
\theoremstyle{plain}
\newtheorem{thm}{Theorem}[section]
\newtheorem{lem}[thm]{Lemma}
\newtheorem{cor}[thm]{Corollary}
\newtheorem{prop}[thm]{Proposition}
\theoremstyle{definition}
\newtheorem{defn}[thm]{Definition}

\theoremstyle{remark}
\newtheorem{rem}[thm]{Remark}
\newcommand{\Thm}[1]{Theorem~\ref{th:#1}}
\newcommand{\Prop}[1]{Proposition~\ref{prop:#1}}
\newcommand{\Lem}[1]{Lemma~\ref{lem:#1}}
\newcommand{\Cor}[1]{Corollary~\ref{cor:#1}}
\newcommand{\Defn}[1]{Definition~\ref{def:#1}}
\newcommand{\Rem}[1]{Remark~\ref{rem:#1}}

\newcommand{\Eq}[1]{\eqref{eq:#1}}
\numberwithin{equation}{section}\makeatletter
\def\rom#1{\mbox{\leavevmode\skip@\lastskip\unskip\/\ifdim\skip@=\z@\else\hskip\skip@\fi{\rm{#1}}}}
\makeatother

\renewcommand{\a}{\alpha}\renewcommand{\b}{\beta}
\newcommand{\gm}{\gamma}\newcommand{\dl}{\delta}
\newcommand{\eps}{\varepsilon}

\newcommand{\lm}{\lambda}
\newcommand{\sg}{\sigma}
\newcommand{\ph}{\varphi}

\newcommand{\Ph}{\Phi}

\newcommand{\D}{\mathbb{D}}

\newcommand{\N}{\mathbb{N}}

\newcommand{\R}{\mathbb{R}}

\newcommand{\cE}{\mathcal{E}}
\newcommand{\cF}{\mathcal{F}}
\newcommand{\fB}{\mathfrak{B}}
\newcommand{\fM}{\mathfrak{M}}
\newcommand{\sd}{\mathsf{d}}
\newcommand{\hH}{\mathscr{H}}
\newcommand{\FCb}{\mathcal{F}C_b}
\newcommand{\la}{\langle}
\newcommand{\ra}{\rangle}

\newcommand{\Dom}{\operatorname{Dom}}
\newcommand{\Cp}{\operatorname{Cap}\nolimits}

\newcommand{\BBG}{\overline{B}_G}
\newcommand{\BBH}{\overline{B}_H}
\DeclareMathOperator{\mint}{%
\mathchoice%
{%displaystyle
\ooalign{%
\ensuremath{%
\begin{picture}(8,8)
\thicklines
\put(1.25,3.0){\line(1,0){8}}
\end{picture}}%
\crcr
\hss\ensuremath{\displaystyle\int}\hss}
\hspace*{-4.25pt}
}%end of displaystyle
{%textstyle
\ooalign{%
\ensuremath{%
\begin{picture}(8,8)
\put(2.0,2.75){\line(1,0){4}}
\end{picture}}%
\crcr
\hss\ensuremath{\textstyle\int}\hss}
\hspace*{-2.63pt}
}%end of textstyle
{%scriptstyle
\ooalign{%
\ensuremath{%
\begin{picture}(8,8)
\thinlines
\put(2.5,2.0){\line(1,0){3}}
\end{picture}}%
\crcr
\hss\ensuremath{\scriptstyle\int}\hss}
\hspace*{-3.20pt}
}%end of scriptstyle
{%scriptscriptstyle
\ooalign{%
\ensuremath{%
\begin{picture}(8,8)
\put(2.98,1.5){\line(1,0){2}}
\end{picture}}%
\crcr
\hss\ensuremath{\scriptscriptstyle\int}\hss}
\hspace*{-3.0pt}
}%end of scriptscriptstyle
}%end of #2 of DeclareMathOperator
\begin{document}
\begin{frontmatter}



\title{Dirichlet spaces on $H$-convex sets in Wiener space}


\author{Masanori Hino\corref{cor1}}
\cortext[cor1]{Research partially supported by KAKENHI (21740094).}
\address{Graduate School of Informatics,
Kyoto University,
Kyoto 606-8501,
Japan
\bigskip

{\normalsize Dedicated to the memory of Professor Paul Malliavin}}

\begin{abstract}
We consider the $(1,2)$-Sobolev space $W^{1,2}(U)$ on subsets $U$ in an abstract Wiener space, which is regarded as a canonical Dirichlet space on $U$.
We prove that $W^{1,2}(U)$ has smooth cylindrical functions as a dense subset if $U$ is $H$-convex and $H$-open.
For the proof, the relations between $H$-notions and quasi-notions are also studied.
\end{abstract}


\end{frontmatter}

% \linenumbers

\section{Introduction}
In  Euclidean space, extension operators related to Sobolev spaces are useful tools.
Their existence is stated as follows: Given a domain $U$ of $\R^n$ with a sufficiently regular boundary, $p\ge1$, and $r\in\N$, there exists a bounded linear map $T\colon W^{r,p}(U)\to W^{r,p}(\R^n)$ such that $Tf=f$ on $U$ for all $f\in W^{r,p}(U)$.
Here, $W^{r,p}(X)$ denotes the Sobolev space on domain $X$, with differentiability index $r$ and integrability index $p$.
In particular, the above statement implies that $W^{r,p}(\R^n)|_U=W^{r,p}(U)$, where the left-hand side denotes a function space on $U$ that is defined by restricting the defining sets of the functions in $W^{r,p}(\R^n)$ to $U$. Hereafter, we use the standard notation described above.
Such properties can reduce many problems on $U$ to those on $\R^n$, which are often easier to resolve.

In this paper, we discuss a related problem in infinite-dimensional spaces.
To the best of the author's knowledge, there are no nontrivial examples that involve the existence of the extension operators described above: Some useful techniques such as covering arguments and a Whitney decomposition in Euclidean space are not directly available in infinite dimensions; this complicates the problem.
In this paper, we consider a reduced version of the problem as follows: Let $(E,H,\mu)$ be an abstract Wiener space (the definition of which is provided in Section~2) and $U$, a measurable subset of $E$ with positive $\mu$-measure.
Find sufficient conditions on $U$ such that 
\begin{equation}\label{eq:property}
\begin{array}{l}
W^{1,2}(E)|_U\mbox{ is dense in $W^{1,2}(U)$ in the topology induced by the}\\\mbox{Sobolev norm.}
\end{array}
\end{equation}
The well-definedness of the space $W^{1,2}(U)$ is explained in the next section. 
Here, we note that $W^{1,2}(U)$ is regarded as the domain of a canonical Dirichlet form on $L^2(U,\mu|_U)$, where $\mu|_U(\cdot):=\mu(\cdot\cap U)$.
Since the space $\FCb^1(E)$ of smooth cylindrical functions on $E$ is known to be dense in $W^{1,2}(E)$, \Eq{property} is equivalent to the following:
\begin{equation}\label{eq:property2}
\begin{array}{l}
\FCb^1(E)|_U\mbox{ is dense in $W^{1,2}(U)$ in the topology induced by the}\\\mbox{Sobolev norm.}
\end{array}
\end{equation}
The closure of $\FCb^1(E)|_U$ in $W^{1,2}(U)$ is often regarded as the minimal domain.
Therefore, the problem under consideration is to determine whether the canonical domain and the minimal domain coincide.
Even for this weaker property, few examples of non-smooth sets are known to satisfy it.
The following is the known result.
\begin{thm}[{\cite[Theorem~2.2]{Hi03}}]\label{th:before}
If $U$ is convex and has a nonempty interior, then \Eq{property} is true.
\end{thm}
We may assume that $U$ is open in this theorem without loss of generality because the topological boundary of $U$ is a $\mu$-null set under these assumptions (see \Rem{1}~(ii)).

In this paper, we provide a refinement of \Thm{before}.
Although \Thm{before} has been proved within a more general framework~\cite{Hi03},  we consider only an abstract Wiener space in order to avoid inessential technical issues.
\Thm{before} is not satisfactory in that the assumptions involve the vector space structure and topological structure of $E$.
It is desirable to impose assumptions depending only on the structures of the Cameron--Martin space $H$.
Accordingly, we prove the following theorem.
Let $\fM(E)$ denote the completion of the Borel $\sg$-field $\fB(E)$ of $E$ by $\mu$.
\begin{thm}\label{th:main}
Suppose that $U\in\fM(E)$ with positive $\mu$-measure is $H$-convex and $H$-open. Then, \Eq{property} holds.
\end{thm}
Here, $U$ is called $H$-convex if $(U-z)\cap H$ is convex in $H$ for all $z\in E$, and $U$ is called $H$-open if $(U-z)\cap H$ has $0$ as an interior point in $H$ for every $z\in U$.
Since convexity in $E$ implies $H$-convexity, and open sets in $E$ are $H$-open, the assumptions of \Thm{main} are essentially weaker than those of \Thm{before}.

If $E$ is finite-dimensional, \Thm{main} is easy to prove as follows: For simplicity, we further assume that $U$ is bounded and contains $0$.
For a small positive number $\gm$, we consider a contraction map $T_\gm\colon E\ni z\mapsto(1-\gm)z\in E$ and let $U_\gm:=T_\gm^{-1}(U)\supset U$.
For each $f\in W^{1,2}(U)$, the function $f\circ T_\gm$, denoted by $f_\gm$, is defined on $U_\gm$ and $f_\gm|_U$ approximates $f$ in $W^{1,2}(U)$.
Since there is a positive distance between $E\setminus U_{\gm/2}$ and the closure of $U$, we can take a Lipschitz function $\ph$ on $E$ such that $\ph=1$ on $U$ and $\ph=0$ on $E\setminus U_{\gm/2}$.
Then, $\ph f_\gm$ is well-defined as a function in $W^{1,2}(E)$ and $(\ph f_\gm)|_U=f_\gm|_U$, which deserves to be an approximating function of $f$.

This proof breaks down when $E$ is infinite-dimensional, since measures $\mu$ and $\mu\circ T_\gm^{-1}$ are mutually singular.
Therefore, our strategy is decomposing $E$ into a finite-dimensional space and an auxiliary space, and applying the procedure stated above for each finite-dimensional section.
\Thm{before} was proved in this way.
The proof of \Thm{main} is similar to that of \Thm{before}, but more technically involved.
This is because we have to treat the topology of $E$ induced by the $H$-distance, which is neither metrizable nor second-countable; we cannot utilize the general theory of good topological spaces.
In order to overcome this difficulty, we firstly study the relations between $H$-notions and quasi-notions, and we use them for removing a suitable set with small capacity to adopt a method utilized in the proof of \Thm{before}.

The remainder of this paper is organized as follows. In Section~2, we introduce a framework, and we prove some preliminary results that are of contextual interest. Some results may be known to experts; nonetheless, we provide proofs of the claims for which the author could not find suitable references. In Section~3, we prove \Thm{main}.
In Section~4, we discuss some applications.
\section{Framework and preliminary propositions}
Let $(E,H,\mu)$ be an abstract Wiener space.
That is, $E$ is a real Banach space with norm $|\cdot|_E$, $H$ is a real separable Hilbert space that is continuously embedded in $E$, and $\mu$ is a Gaussian measure on $E$ such that
\[
  \int_E \exp\left(\sqrt{-1}\,l(z)\right)\mu(dz)=\exp\left(-|l|_H^2/2\right)\quad
  \mbox{for all }l\in E^*\subset H^*\simeq H\subset E.
\]
Here, we denote the topological duals of $E$ and $H$ by $E^*$ and $H^*$, respectively, and we adopt the inclusions and identification stated above.
We always assume that $E$ is infinite-dimensional. 
The inner product and norm of $H$ are denoted by $\la\cdot,\cdot\ra$ and $|\cdot|_H$, respectively.
For $l\in E^*$ and $z\in E$, $l(z)$ also denotes $\la l,z\ra$.
This terminology is consistent with the inner product of $H$ when the inclusions $E^*\subset H\subset E$ are taken into consideration.
For $s\in\R$, $z\in E$, $A\subset E$, and $B\subset E$, we set $sA=\{sa\mid a\in A\}$ and $A\pm B=\{a\pm b\mid a\in A,\ b\in B\}$, and we denote $\{z\}\pm A$ by $z\pm A$.
The following is a basic property of $\mu$ (see, e.g., \cite[Corollary~2.5.4]{Bo} for the proof).
\begin{prop}\label{prop:ergodicity}
Let $A\in\fM(E)$ and $F$ be a dense linear space of $H$.
If $A+F=A$, either $\mu(A)=0$ or $\mu(E\setminus A)=0$ holds.
\end{prop}

For $X\in\fM(E)$, the $\sg$-field $\{A\in\fM(E)\mid A\subset X\}$ on $X$ is denoted by $\fM(X)$.
Given $X\in\fM(E)$, a separable Hilbert space $\hH$, and $p\in[1,+\infty]$, the $\hH$-valued $L^p$-space on the measure space $(X,\fM(X), \mu|_X)$ is denoted by $L^p(X\to \hH)$.
When $\hH=\R$, it is simply denoted by $L^p(X)$.
Its canonical norm is denoted by $\|\cdot\|_{L^p(X)}$.

The space $\FCb^1(E)$ of smooth cylindrical functions on $E$ is defined as
\[
\FCb^1(E)=\left\{u\colon E\to\R\left|\,\begin{array}{ll} u(z)=f(\la l_1,z\ra,\ldots,\la l_m,z\ra),\ l_1,\ldots,l_m\in E^*,\\ f\in C_b^1(\R^m)\text{ for some }m\in\N \end{array}\right.\!\!\!\right\},
\]
where $C_b^1(\R^m)$ is the set of all bounded $C^1$-functions on $\R^m$ with bounded first order derivatives.
Let $G$ be a finite-dimensional subspace of $E^*$.
We define a closed subspace $G^\perp$ of $E$ as
$G^\perp=\{z\in E\mid \la h,z\ra=0 \mbox{ for every }h\in G\}$.
The direct sum $G^\perp\dotplus G$ is identified with $E$.
 The canonical projection maps from $E$ to $G^\perp$ and $G$ are denoted by $P_G$ and $Q_G$, respectively.
To be precise, they are defined as follows:
\[
  P_G z=z-Q_G z,\quad
  Q_G z=\sum_{i=1}^m \la h_i,z\ra h_i,
\]
where $\{h_i\}_{i=1}^m\subset G\subset H\subset E$ is an orthonormal basis of $G$ in $H$.
The image measures of $\mu$ by $P_G$ and $Q_G$ are denoted by $\mu_{G^\perp}$ and $\mu_G$, respectively.
Both measures are centered Gaussian measures; in particular, $\mu_G$ is described as
\[
  \mu_G(dy)=(2\pi)^{-m/2}\exp(-|y|_H^2/2)\,\lm_m(dy),
\]
where $m=\dim G$ and $\lm_m$ denotes the Lebesgue measure on $G$.
The product measure of $\mu_{G^\perp}$ and $\mu_G$ is identified with $\mu$.
When $G=\R h$ for some $h\in E^*$, we write $h^\perp$, $\mu_{h^\perp}$, and $\mu_h$ for $G^\perp$,  $\mu_{G^\perp}$, and $\mu_G$, respectively.

Let $X\in\fM(E)$.
For $h\in E^*\setminus\{0\}\subset E$ and $x\in h^\perp$, we define
\[
  I^X_{x,h}=\{s\in\R\mid x+sh\in X\}.
\]
We fix a linear subspace $K$ of $E^*$ that is dense in $H$.
We call $X$ {\em $K$-moderate} if for each $h\in K\setminus\{0\}$, the boundary of $I^X_{x,h}$ in $\R$ is a Lebesgue null set for $\mu^\perp$-a.e.\,$x\in h^\perp$.
It is evident that $H$-convex sets in $\fM(E)$ are $K$-moderate.
For a function $f$ on $X$, $x\in E$, and $h\in E^*\setminus\{0\}$, we define a function $f_h(x,\cdot)$ on $I^X_{x,h}$ as $f_h(x,s)=f(x+sh)$.

Suppose that $X$ is $K$-moderate and $\mu(X)>0$.
For $h\in K\setminus\{0\}$, let $\Dom(\cE^X_h)$ be the set of all functions $f$ in $L^2(X)$ such that the following hold:
\begin{itemize}
\item For $\mu_{h^\perp}$-a.e.\,$x\in h^\perp$, $f_h(x,\cdot)$ has an absolutely continuous version $\tilde f_h(x,\cdot)$ on the interior of the closure of $I^X_{x,h}$ in $\R$.
\item There exists an element of $L^2(X)$, denoted by $\partial_h f$,  such that for $\mu_{h^\perp}$-a.e.\,$x\in h^\perp$,
$(\partial_h f)(x+sh)=\frac{\partial \tilde f_h}{\partial s}(x,s)$ for a.e.\,$s\in I^X_{x,h}$ with respect to the one-dimensional Lebesgue measure.
\end{itemize}
Then, the bilinear form $(\cE^X_h,\Dom(\cE^X_h))$ on $L^2(X)$, defined as
\[
  \cE^X_h(f,g)=\int_X (\partial_h f)(\partial_h g)\,d\mu,\quad
  f,g\in\Dom(\cE^X_h),
\]
is a closed form from \cite[Theorem~3.2]{AR90}. 
The $(1,2)$-Sobolev space $W^{1,2}(X)$ on $X$ is then defined as
\[
  W^{1,2}(X)=\left.\left\{f\in\ \bigcap_{\mathclap{h\in K\setminus\{0\}}}\ \Dom(\cE^X_h)\,\right|\,
   \parbox{0.55\textwidth}{%
  there exists $Df\in L^2(X\to H)$ such that
  $\la Df,h\ra=\partial_h f$ $\mu$-a.e.\ on $X$ for every $h\in K\setminus\{0\}$}\right\}.
\]
Space $W^{1,2}(X)$ formally corresponds to the maximal domain in the terminology of \cite{AKR90} and the weak Sobolev space in that of \cite{Eb}, even though the validity of these terminologies have not been investigated in our situation because our framework does not satisfy the conditions in the corresponding theorems in \cite{AKR90,Eb}.
The bilinear form $(\cE^X, W^{1,2}(X))$ on $L^2(X)$, defined as
\[
  \cE^X(f,g)=\int_X\la Df,Dg\ra\,d\mu,\quad
  f,g\in W^{1,2}(X),
\]
is a local Dirichlet form in terms of \cite[Definition~I.5.1.2]{BH}.
We note some properties for future reference.
\begin{prop}[{cf.~\cite{BH}, \cite[Proposition~2.1]{Hi03}}]\label{prop:derivation}
  Let $\Phi$ be a Lipschitz function on $\R$ and let $f$ and $g$ be functions in $W^{1,2}(X)$. Then:
  \begin{enumerate}
    \item For any Lebesgue null set $A$ of $\R$, $D f=0$ $\mu$-a.e.\ on $f^{-1}(A)$.
    In particular, if $f=0$ on a measurable set $B$, then $D f=0$ $\mu$-a.e.\ on $B$.
    \item $\Phi(f)\in W^{1,2}(X)$ and $D(\Phi(f))=\Phi'(f) Df$ $\mu$-a.e.
    \item In addition, if $f,g\in L^\infty(X)$, then $fg\in W^{1,2}(X)$ and $D(fg)=f(Dg)+g(Df)$ $\mu$-a.e.
    \end{enumerate}
\end{prop}

 We write $\cE$ for $\cE^E$.
The norm $\|\cdot\|_{W^{1,2}(X)}$ of $W^{1,2}(X)$ is given by $\|f\|_{W^{1,2}(X)}=\left(\cE^X(f,f)+\|f\|_{L^2(X)}^2\right)^{1/2}$.
In general, although $W^{1,2}(X)$ may depend on the choice of $K$, we omit the dependency on $K$ from the notation for simplicity.
It is known that $W^{1,2}(E)$ does not depend on the choice of $K$ and includes $\FCb^1(E)$ as a dense subset in the topology induced by $\|\cdot\|_{W^{1,2}(E)}$.
Therefore, under the assumptions on $U$ in \Thm{main}, conclusion~\Eq{property2} implies {\em a posteriori} that $W^{1,2}(U)$ is independent of the choice of $K$.

We now recall the concepts of capacity and the associated quasi-notions.
Since we use these terminologies with respect to only $(\cE,W^{1,2}(E))$, we define them in this particular case.
For open subsets $O$ of $E$, the capacity of $O$ (with respect to $(\cE,W^{1,2}(E))$) is defined as
\[
\Cp_{1,2}(O):=\inf\{\cE(f,f)+\|f\|_{L^2(E)}^2\mid f\in W^{1,2}(E)\text{ and $f\ge1$ $\mu$-a.e.\ on $O$}\}.
\]
The infimum stated above is attained by a unique function $e_O$, known as the equilibrium potential of $O$. It holds that $0\le e_O\le1$ $\mu$-a.e.\ and $e_O=1$ $\mu$-a.e.\ on $O$.
For a general subset $A$ of $E$, its capacity is defined as 
\[
\Cp_{1,2}(A):=\inf\{\Cp_{1,2}(O)\mid \text{$O$ is open and $O\supset A$}\}.
\]
We remark that $\Cp_{1,2}$ is countably subadditive.
A function $f$ on $E$ is called quasi-continuous if for any $\eps>0$, there exists an open set $O$ such that $\Cp_{1,2}(O)<\eps$ and $f|_{E\setminus O}$ is continuous on $E\setminus O$.
Since $(\cE,W^{1,2}(E))$ is quasi-regular (see \cite{MR} for the definition), each element of $W^{1,2}(E)$ has a quasi-continuous modification.
A subset $A$ of $E$ is called quasi-closed if for any $\eps>0$, there exists an open set $O$ such that $\Cp_{1,2}(O)<\eps$ and $A\setminus O$ is closed. A subset $A$ is called quasi-open if $E\setminus A$ is quasi-closed. 
For two functions $f$ and $g$ on $E$, we write $f=g$ q.e.\ if $\Cp_{1,2}(\{f\ne g\})=0$.

A subset $E_0$ of $E$ is called $H$-invariant if $E_0+H=E_0$.
\begin{defn}\label{def:HL}
Let $f$ be a $[-\infty,+\infty]$-valued function on $E$.
\begin{enumerate}
\item $f$ is called $H$-continuous if there exists an $H$-invariant set $E_0$ such that $\mu(E\setminus E_0)=0$, $|f(z)|<\infty$ for every $z\in E_0$, and the function $f(z+\cdot)$ on $H$ is continuous in the topology of $H$ for every $z\in E_0$.\footnote{This definition may slightly differ from those in other literatures.}
\item $f$ is called $H$-Lipschitz if there exist an $H$-invariant set $E_0$ and a constant $M\ge0$ such that $\mu(E\setminus E_0)=0$, $|f(z)|<\infty$ for every $z\in E_0$, and $|f(w)-f(z)|\le M|w-z|_H$ for all $w,z\in E_0$. In this case, we say that $f$ has $H$-Lipschitz constant (at most) $M$.
\end{enumerate}
\end{defn}
The following is a variant of Rademacher's theorem.
\begin{thm}[{\cite{ES93}, cf.\ \cite[Theorem~4.2]{Ku82}}]\label{th:rademacher}
Let $f$ be an $\fM(E)$-measurable function on $E$ that is $H$-Lipschitz with $H$-Lipschitz constant $M$. Then, $f\in W^{1,2}(E)$ and $|Df|_H\le M$ $\mu$-a.e.
\end{thm}
We introduce some concepts related to the $H$-distance.
\begin{defn}\label{def:H}
For a subset $A$ of $E$ and $z\in E$, we define
\[
\sd_E(z,A)=\inf\{|z-w|_E\mid w\in A\}\text{ and }
\sd_H(z,A)=\inf\{|z-w|_H\mid w\in A\cap (z+H)\},
\]
where we set $\inf\emptyset=+\infty$.
We also define the following sets:
\begin{itemize}
\item the $H$-closure $\overline{A}^H:=\{z\in E\mid \sd_H(z,A)=0\}$, 
\item the $H$-boundary $\partial^H A:=\overline{A}^H\cap\overline{E\setminus A}^H$,
\item the $H$-exterior $A^{H\text{-ext}}:=E\setminus\overline{A}^H$,
\item the $H$-interior $A^{H\text{-int}}:=A\setminus\partial^H A\,(=(A^{H\text{-ext}})^{H\text{-ext}})$.
\end{itemize}
For $z\in E$ and $s>0$, we define
\[
  B_H(z,s)=\{z+h\mid h\in H,\ |h|_H<s\}\text{ and }
  \BBH(z,s)=\{z+h\mid h\in H,\ |h|_H\le s\}.
\]
\end{defn}
We omit $z$ from the notation if $z=0$.
Note that $\BBH(z,s)$ is compact in $E$ (see, e.g., \cite[Corollary~3.2.4]{Bo} for the proof.)

Let us recall that a Suslin set in $E$ is a continuous image of a certain Polish space. Suslin sets are universally measurable and closed under countable intersections and countable unions. Borel sets of $E$ are Suslin sets.  
More precisely speaking, a subset $A$ of $E$ is Borel if and only if both $A$ and $E\setminus A$ are Suslin sets (see, e.g., \cite{Bou,DM} for further details).
\begin{lem}\label{lem:distance}
If $A$ is a Suslin subset of $E$ with $\mu(A)>0$, then $\sd_H(\cdot,A)$ is universally measurable and $H$-Lipschitz with $H$-Lipschitz constant $1$.
\end{lem}
\begin{proof}
Measurability of $\sd_H(\cdot,A)$ follows from the identity
\[
\{\sd_H(\cdot,A)\le r\}=\bigcap_{k=1}^\infty \left(A+\BBH(r+1/k)\right)
\]
for every $r\ge0$.
The set $\{\sd_H(\cdot,A)<\infty\}$ has a full $\mu$-measure from \Prop{ergodicity}.
The remaining assertions are easy to prove.
\end{proof}
The next proposition is proved in \cite{RR05} in a more general context. (Similar results are found, e.g., in \cite{RSc99} in different frameworks.) 
Since our situation is simpler and the proof is shortened, we include the proof for the readers' convenience.
\begin{prop}\label{prop:qe}
Suppose that a subset $A$ of $E$ is $H$-open and $\mu(A)=0$.
Then, $\Cp_{1,2}(A)=0$.
\end{prop}
\begin{proof}
We denote $E\setminus A$ by $A^c$.
Take any $\eps>0$ and a compact subset $C$ of $A^c$ such that $\mu(E\setminus C)<\eps$.
Let $r>0$ and define $C_r:=C+\BBH(r)$, which is a compact set.
Define $f_r(z)=\left(r^{-1}\sd_H(z,C)\right)\wedge1$ for $z\in E$.
Then, $f_r\in W^{1,2}(E)$ and $|Df_r|_H\le 1/r$ $\mu$-a.e.\ from \Lem{distance}, \Thm{rademacher}, and \Prop{derivation}.
Moreover, $0\le f_r\le 1$ on $E$, $f_r=1$ on $E\setminus C_r$, and $f_r=0$ on $C$.
Then,
\begin{align*}
\Cp_{1,2}\left(E\setminus\left(A^c+\BBH(r)\right)\right)
&\le \Cp_{1,2}(E\setminus C_r)
\le \cE(f_r,f_r)+\|f_r\|_{L^2(E)}^2\\
&\le \int_{E\setminus C} r^{-2}\,d\mu+\int_{E\setminus C}d\mu
\quad \text{(by \Prop{derivation} (i))}\\
&\le (r^{-2}+1)\eps.
\end{align*}
Since $\eps$ is arbitrary, $\Cp_{1,2}\left(E\setminus\left(A^c+\BBH(r)\right)\right)=0$.
Therefore, 
\[
\Cp_{1,2}(A)=\Cp_{1,2}\left(\bigcup_{k=1}^\infty \left(E\setminus\left(A^c+\BBH(1/k)\right)\right)\right)=0.\qedhere
\]
\end{proof}
\begin{cor}[{cf.\ \cite[Theorem~7.3.3]{Us}}]\label{cor:cap01}
Let $A\in \fM(E)$ be $H$-invariant.
Then, either $\Cp_{1,2}(A)=0$ or $\Cp_{1,2}(E\setminus A)=0$ holds.
\end{cor}
\begin{proof}
Since both $A$ and $E\setminus A$ are $H$-open, the assertion follows from Propositions~\ref{prop:ergodicity} and \ref{prop:qe}.
\end{proof}
\begin{lem}\label{lem:Lip}
Let $f$ be an $\fM(E)$-measurable and $H$-Lipschitz function on $E$.
Then, $f$ is quasi-continuous.
\end{lem}
From \Thm{rademacher}, $f$ belongs to $W^{1,2}(E)$ under the assumption; thus, $f$ has a quasi-continuous modification.
The point of \Lem{Lip} is that $f$ itself is quasi-continuous without modification.
\begin{proof}[Proof of \Lem{Lip}]
From \Prop{qe}, $\Cp_{1,2}(E\setminus E_0)=0$, where $E_0$ is provided in \Defn{HL}.
Therefore, by considering $f\cdot 1_{E_0}$ instead of $f$, we may assume $E_0=E$ without loss of generality.
Let $f$ have $H$-Lipschitz constant $M$.

We take an increasing sequence $\{G_n\}_{n=1}^\infty$ of finite-dimensional subspaces of $E^*$ such that $\bigcup_{n=1}^\infty G_n$ is dense in $H$.
We also define $G_n^\perp$, $P_{G_n}$, $Q_{G_n}$, $\mu_{G_n^\perp}$, and $\mu_{G_n}$ as in the first part of this section.
For $n\in\N$, let 
\[
\hat G_n=\left\{y\in G_n\left|\, \text{$f(\cdot+y)$ is a $\mu_{G_n^\perp}$-integrable function on $G_n^\perp$}\right\}\right..
\]
From Fubini's theorem, $\mu_{G_n}(G_n\setminus \hat G_n)=0$.
Define a function $\hat f_n$ on $\hat G_n$ by 
\[
 \hat f_n(y)=
 \int_{G_n^\perp}f(x+y)\,\mu_{G_n^\perp}(dx).
\]
Then, it is easy to see that for $y,y'\in \hat G_n$,
\begin{equation}\label{eq:Gn}
  |\hat f_n(y)-\hat f_n(y')|\le M|y-y'|_H.
\end{equation}
Therefore, $\hat f_n$ extends to a continuous function $\hat{\hat {f_n}}$ that is defined on $G_n$, and \Eq{Gn} holds for every $y,y'\in G_n$ with $\hat f_n$ replaced by $\hat{\hat f}_n$.
Define a function $f_n$ on $E$ as $f_n(z)=\hat{\hat {f_n}}(Q_n(z))$ for $z\in E$.
Then, $f_n$ is continuous on $E$ and identical to the conditional expectation of $f$ given $\sg(Q_n)$.
Since $\sg(Q_n; n\in\N)=\fB(E)$, $f_n$ converges to $f$ $\mu$-a.e.\ by the martingale convergence theorem.
Moreover, since $Q_n|_H$ is a contraction operator on $H$, $f_n$ is also $H$-Lipschitz with $H$-Lipschitz constant $M$.
Then, $\{f_n\}_{n=1}^\infty$ is bounded in $W^{1,2}(E)$.
From the Banach--Saks theorem, the Ces\`aro means of a certain subsequence of $\{f_n\}$, denoted by $\{g_n\}$, converge in $W^{1,2}(E)$.
Note that $g_n$ is continuous on $E$ as well as $H$-Lipschitz with $H$-Lipschitz constant $M$.
From \cite[Proposition~III.3.5]{MR} or \cite[Theorem~2.1.4]{FOT}, by taking a subsequence if necessary, $g_n$ converges q.e.\ to some quasi-continuous function $g$.
Since $f_n$ converges to $f$ $\mu$-a.e., so does $g_n$.
Define $B=\{z\in E\mid \lim_{n\to\infty}g_n(z)=f(z)\}$.
Clearly, $\mu(E\setminus B)=0$.
Take $z\in \overline{B}^H$.
There exists a sequence $\{z_k\}_{k=1}^\infty$ in $B$ such that $\lim_{k\to\infty}|z_k-z|_H=0$.
Then,
\begin{align*}
|g_n(z)-f(z)|
&\le |g_n(z)-g_n(z_k)|+|g_n(z_k)-f(z_k)|+|f(z_k)-f(z)|\\
&\le M|z-z_k|_H+|g_n(z_k)-f(z_k)|+M|z-z_k|_H.
\end{align*}
Taking $\limsup_{n\to\infty}$ on both sides and letting $k\to\infty$, we obtain $\lim_{n\to\infty}g_n(z)=f(z)$.
Therefore, $z\in B$. That is, $\overline{B}^H=B$ and $E\setminus B$ is $H$-open.
From \Prop{qe}, $\Cp_{1,2}(E\setminus B)=0$.
This implies that $f=g$ q.e., in particular, $f$ is quasi-continuous.
\end{proof}
The following proposition, which is of contextual interest, is utilized in the proof of \Thm{main} in the next section.
\begin{prop}\label{prop:m}
Let $A\in\fM(E)$.
\begin{enumerate}
\item If $A$ is $H$-open, then $A^{H\text{\rm -ext}}$ is quasi-open; in particular, $A^{H\text{\rm -ext}},\overline{A}^H,\partial^H A\in\fM(E)$.
\item If $A$ is $H$-open, then $A$ is quasi-open.
\item If $A$ is $H$-closed, then $A$ is quasi-closed.
\end{enumerate}
\end{prop}
\begin{proof}
(i): If $\Cp_{1,2}(\overline{A}^H)=0$, then the assertion is clear.
We assume $\Cp_{1,2}(\overline{A}^H)>0$.
Choose a countable dense subset $H_0$ of $H$.
From the $H$-openness of $A$, $\sd_H(z,A)=\inf\{|h|_H\cdot 1_{A-h}(z)\mid h\in H_0\}$ for each $z\in E$.
Thus, $\sd_H(\cdot,A)$ is $\fM(E)$-measurable. 
Let $E_0=\{z\in E\mid \sd_H(z,A)<\infty\}$.
Since $E_0$ is $H$-invariant and $E_0\supset \overline{A}^H$, $\Cp_{1,2}(E\setminus E_0)=0$ from \Cor{cap01}. 
In particular, $\mu(E_0)=1$. Therefore, $\sd_H(\cdot,A)$ satisfies the definition of $H$-Lipschitz functions. From \Lem{Lip}, it is quasi-continuous.
Since $A^{H\text{-ext}}=\{z\in E\mid \sd_H(z,A)>0\}$, $A^{H\text{-ext}}$ is quasi-open. 

(ii): By applying (i) to the $H$-open set $A^{H\text{-ext}}$ and by using the identity $A=(A^{H\text{-ext}})^{H\text{-ext}}$, we conclude that $A$ is quasi-open.

(iii): It is sufficient to apply (ii) to $E\setminus A$.
\end{proof}
The following is an improvement on \Lem{Lip}.
\begin{prop}\label{prop:Hqe}
If an $\fM(E)$-measurable function $f$ on $E$ is $H$-continuous, then $f$ is quasi-continuous.
\end{prop}
\begin{proof}
This is clear from \Cor{cap01} and \Prop{m}.
\end{proof}
\begin{rem}
From the proof, we can replace $W^{1,2}(E)$ by $W^{1,p}(E)$ in \Thm{rademacher} and $\Cp_{1,2}$ by $\Cp_{1,p}$ in \Prop{qe} and \Cor{cap01} for any $p\in (1,\infty)$. 
Here, $W^{1,p}(E)$ is the first order $L^p$-Sobolev space on $E$ in terms of Malliavin calculus, and $\Cp_{1,p}$ is the associated capacity.
(See \cite{Ma} for example, where symbol $\D^p_1$ is used in place of $W^{1,p}$.)
Moreover, \Lem{Lip}, \Prop{m}, and \Prop{Hqe} are valid, even if quasi-notions are interpreted in terms of $\Cp_{1,p}$.
\end{rem}
\section{Proof of \Thm{main}}
In this section, we prove \Thm{main}.
We assume that $U\in \fM(E)$ satisfies the assumptions of \Thm{main}: $\mu(U)>0$, $U$ is $H$-open and $H$-convex.
For a subset $F$ of $E$ and subset $A$ of $F$, we denote the closure, interior, and boundary of $A$ with respect to the relative topology of $F$ by $\overline{A}^F$, $A^{F\text{-int}}$, and $\partial^F A$, respectively.
Although these terminologies are slightly inconsistent with the corresponding ones in \Defn{H}, we use them as long as there is no ambiguity.

Let us recall that $K$ was taken and fixed as a dense subspace of $H$ in Section~2.
We also fix an increasing sequence $\{G_n\}_{n=1}^\infty$ of finite-dimensional subspaces of $K$ such that $\bigcup_{n=1}^\infty G_n$ is dense in $H$.
We also define $G_n^\perp$, $P_{G_n}$, $Q_{G_n}$, $\mu_{G_n^\perp}$, and $\mu_{G_n}$ as in the previous section.
For a finite-dimensional subspace $G$ of $K$ and $x\in G^\perp$, $\mu_{x+G}$ denotes a measure on $x+G$ that is defined as the induced measure of $\mu_G$ by the canonical map from $G$ to $x+G$.

The following is a consequence of the basic theory of convex analysis; it is proved in the same way as in \cite[Lemma~4.7]{Hi10}.
\begin{lem}\label{lem:section}
Let $G$ be a finite-dimensional subspace of $H$ and $a\in E$.
Define $F=a+G$.
If $U\cap F\ne\emptyset$, then $U^{H\text{\rm -int}}\cap F=(U\cap F)^{F\text{\rm -int}}$, $\overline{U}^H\cap F=\overline{U\cap F}^F$, and $(\partial^H U)\cap F=\partial^F(U\cap F)$.
\end{lem}
\begin{proof}
Select $y$ from $U^{H\text{-int}}\cap F$.
There exists $\dl>0$ such that $B_H(y,\dl)\subset U^{H\text{-int}}$.

First, we show that $U^{H\text{-int}}\cap F\supset (U\cap F)^{F\text{-int}}$.
Take $x$ from $(U\cap F)^{F\text{-int}}$.
There exists $s>0$ such that $w:=(1+s)x-sy\in (U\cap F)^{F\text{-int}}$.
Then, $B_H(x,\frac{s\dl}{1+s})=\frac1{1+s}w+\frac{s}{1+s}B_H(y,\dl)\subset U$, that is, $x\in U^{H\text{-int}}$.
Since $x\in F$, we obtain $x\in U^{H\text{-int}}\cap F$.

Next, we show that $\overline{U}^H\cap F\subset\overline{U\cap F}^F$.
Take $x\in\overline{U}^H\cap F$. Then,
\[
 \bigcup_{t\in(0,1]}\bigl((1-t)x+t(B_H(y,\dl)\cap F)\bigr)\subset(U\cap F)^{F\text{-int}}
\]
(cf.\ \cite[Theorem~6.1]{Ro}), and $x$ is an accumulation point in $F$ of the left-hand side.
Therefore, $x\in\overline{U\cap F}^F$.

Both the converse inclusions are evident. The last identity follows from the first two identities.
\end{proof}
\begin{lem}\label{lem:Kz}
There exist a compact subset $V_0$ of $U$ and $r>0$ such that $\mu(V_0)>0$ and $V_0+\BBH(4r)\subset U$.
\end{lem}
\begin{proof}
In the proof, we do not use the $H$-convexity of $U$.
By taking an open set $O$ of $E$ with $0<\mu(\overline{O}^H)<\mu(U)$ and considering $U\setminus \overline{O}^H$ instead of $U$, we may assume $\mu(U^{H\text{-ext}})>0$.
Define $\ph(z)=\sd_H(z,U^{H\text{-ext}})$ for $z\in E$. Then, $\ph$ is $\fM(E)$-measurable from the proof of \Prop{m}~(i).
Since $U=\{\ph>0\}$, we can take $r>0$ such that $\mu(\{\ph\ge5r\})>0$.
Take a compact subset $V_0$ of $\{\ph\ge5r\}$ such that $\mu(V_0)>0$. 
These satisfy the required conditions.
\end{proof}
Hereafter, $V_0$ and $r$ always denote those in \Lem{Kz}.
We define $V=V_0+\BBH(r)$. 
Note that $V$ is compact and 
\begin{equation}\label{eq:V}
V+\BBH(3r)\subset U.
\end{equation}
\begin{rem}\label{rem:1}
\begin{enumerate}
\item We have $\mu(\partial^H U)=0$.
Indeed, let $x\in P_{G_n}(V)$. 
From \Lem{section}, 
\begin{equation}\label{eq:slice}
(\partial^H U)\cap(x+G_n)=\partial^{x+G_n}(U\cap (x+G_n)). 
\end{equation}
Since $U\cap (x+G_n)$ is convex in $x+G_n$, the right-hand side of \Eq{slice} is a null set with respect to the Lebesgue measure on $x+G_n$, i.e., $\mu_{x+G_n}$-null.
By integrating over $P_{G_n}(V)$, $(\partial^H U)\cap(V+G_n)$ is proved to be a $\mu$-null set.
Since $\mu(E\setminus (V+G_n))\to0$ as $n\to\infty$ from \Prop{ergodicity}, we obtain $\mu(\partial^H U)=0$.
\item Similarly, we can prove that if $U$ is a convex set with nonempty interior in $E$, then the topological boundary of $U$ is a $\mu$-null set.
\end{enumerate}
\end{rem}
\begin{defn}
Let $G$ be a subspace of $H$. For $z\in E$ and $s>0$, we define
\[
B_G(z,s)=\{z+h\mid h\in G,\ |h|_H< s\}\text{ and }
\overline{B}_G(z,s)=\{z+h\mid h\in G,\ |h|_H\le s\}.
\]
\end{defn}
We often omit $z$ from the notation if $z=0$.

Let $W_0$ be a subspace of $W^{1,2}(U)$, defined as follows:
\begin{equation}\label{eq:W1}
W_0=\left\{f\in W^{1,2}(U)\left|\begin{array}{l}
\text{$f$ is bounded on $U$ and $f=0$ $\mu$-a.e.\ on}\\
\text{$U\setminus\left(V+\overline{B}_{G_R}(R)\right)$ for some $R\in\N$}
\end{array}\right.\!\!\!
\right\}.
\end{equation}
\begin{lem}\label{lem:dense1}
Space $W_0$ is dense in $W^{1,2}(U)$.
\end{lem}
\begin{proof}
Since $W^{1,2}(U)\cap L^\infty(U)$ is dense in $W^{1,2}(U)$, it is sufficient to prove that each function in $W^{1,2}(U)\cap L^\infty(U)$ can be approximated by functions in $W_0$.
Take $f\in W^{1,2}(U)\cap L^\infty(U)$ and let $M=\|f\|_{L^\infty(U)}$.
From \Prop{ergodicity}, $\lim_{n\to\infty}\mu\left(V_0+\overline{B}_{G_n}(n)\right)=1$.
For each $n\in\N$, define 
\[
\xi_n(z)=\left(1-r^{-1}\sd_H(z,V_0+\overline{B}_{G_n}(n))\right)\vee0,
\quad z\in E.
\]
From \Thm{rademacher} and \Prop{derivation}, $\xi_n\in W^{1,2}(E)$ and $|D\xi_n|_H\le 1/r$ $\mu$-a.e.
In addition, $0\le \xi_n\le1$ on $E$, $\xi_n=0$ on $E\setminus\left(V+\overline{B}_{G_n}(n)\right)$, and $\xi_n=1$ on $V_0+\overline{B}_{G_n}(n)$.
From \Prop{derivation}, $f\xi_n\in W_0$ and
\begin{align*}
\|f\xi_n\|_{W^{1,2}(U)}^2
&\le 2M^2\|\xi_n\|_{W^{1,2}(U)}^2+2\|f\|_{W^{1,2}(U)}^2+M^2\\
&\le 2M^2(r^{-2}+1)+2\|f\|_{W^{1,2}(U)}^2+M^2,
\end{align*}
which is bounded in $n$. Therefore, the Ces\`aro means of a certain subsequence of $\{f\xi_n\}_{n=1}^\infty$ converges in $W^{1,2}(U)$. Since $\xi_n\to 1$ $\mu$-a.e.\ as $n\to\infty$, the limit function is $f$.
\end{proof}
Hereafter, we fix a function $f$ in $W_0$ and write $G$ for $G_R$ in \Eq{W1}.
For the proof of \Thm{main}, it is sufficient to prove that $f$ is approximated by elements in $W^{1,2}(E)|_U$.
For this purpose, we first construct a partition of unity.

Since $Q_G(V)$ is compact in $G$, we can take a finite number of points $a_1,a_2,\dots,a_S$ from $Q_G(V)$ such that $Q_G(V)\subset \bigcup_{i=1}^S B_G(a_i,r)$ for some $S\in\N$.
For $i=1,\dots, S$, define 
\[
A_i=Q_G^{-1}(\BBG(a_i,r))\cap V\]
and 
\[
\psi_i(z)=\left(1-r^{-1}\sd_H(z,A_i+G)\right)\vee0,
\quad z\in E.
\]
Then, each $A_i$ is compact in $E$, $V+G\subset\bigcup_{i=1}^S (A_i+G)$, $\sum_{i=1}^S \psi_i(z)\ge 1$ for $z\in V+G$, and $\psi_i(z)=0$ for $z\in E\setminus\left(A_i+G+B_H(r)\right)$ for each $i$.

We take a real-valued nondecreasing smooth function $\Ph$ on $\R$ such that $\Ph(0)=0$ and $\Ph(t)=1$ for $t\ge1$.
Define
\[
\ph_1=\Phi(\psi_1),\ \ph_j=\Phi\left(\sum_{i=1}^j\psi_i\right)-\Phi\left(\sum_{i=1}^{j-1}\psi_i\right)\text{ for } j=2,\dots,S.
\]
For each $j$, $\ph_j$ is $H$-Lipschitz, $0\le \ph_j\le1$ on $E$, and $\ph_j=0$ on $E\setminus\left(A_j+G+B_H(r)\right)$.
Moreover, $\sum_{j=1}^S \ph_j=1$ on $V+G$.
Thus, $f\ph_j|_U\in W^{1,2}(U)\cap L^\infty(U)$ for each $j$ and $f=\sum_{j=1}^S f\ph_j$ on $U$.
Therefore, in order to prove \Thm{main}, it is sufficient to prove that each $f\ph_j|_U$ can be approximated by elements in $W^{1,2}(E)|_U$.

We fix $j$ and write $g$ for $f\ph_j|_U$.
\begin{lem}\label{lem:g}
We have $\{g\ne0\}\subset A_j+B_H(r)+\BBG(R')$,
where $R'>0$ is taken such that it is large enough to satisfy 
$Q_G(V-A_j)+B_G(R+r)\subset \overline{B}_G(R')$.
\end{lem}
\begin{proof}
By the definition of $g$, we have
$\{g\ne0\}\subset(V+\BBG(R))\cap(A_j+G+B_H(r))$.
Take an element $z$ from the right-hand side.
Then, $z$ is described as 
\[
z=z_1+y_1=z_2+y_2+h,
\]
where $z_1\in V$, $y_1\in \BBG(R)$, $z_2\in A_j$, $y_2\in G$, and $h\in B_H(r)$. Then,
\[
y_2=Q_G(z_1-z_2+y_1-h)
\in Q_G(V-A_j)+B_G(R+r)
\subset \overline{B}_G(R').
\]
This completes the proof.
\end{proof}
We set 
\begin{equation}\label{eq:Y}
Y=A_j+B_H(r)+\BBG(R'+1),
\end{equation}
which belongs to $\fM(E)$ and is relatively compact as well as $H$-open.
(See Figure~\ref{fig:1}.)
\begin{figure}[htb]
\begin{center}
\begin{overpic}[scale=1]{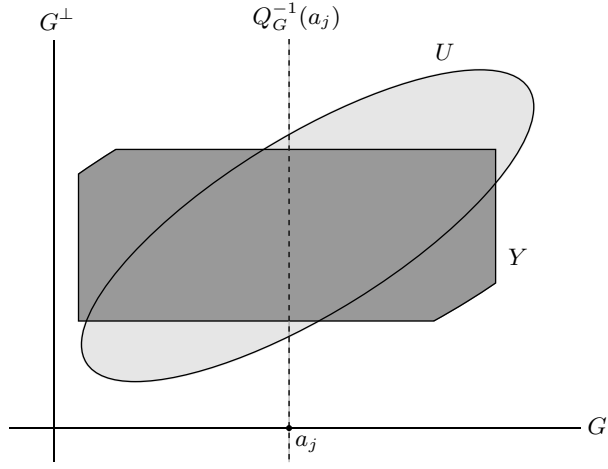}
\put(95,5){\small$G$}
\put(5,72){\small$G^\perp$}
\put(40,73){\small$Q_G^{-1}(a_j)$}
\put(47,3){\small$a_j$}
\put(82,33){\small$Y$}
\put(70,67){\small$U$}
\end{overpic}
\caption{Illustration of $U$ and $Y$ etc.}
\label{fig:1}
\end{center}
\end{figure}
We define 
\[
Y'=(Y+B_H(r))\cap U \quad\text{and}\quad X=\bigl((Q_G^{-1}(a_j)\cap U)+B_G(R'')\bigr)\cap U
\]
with $R''=R'+1+3r$.
\begin{lem}
It holds that $Y\cap U\subset Y'\subset X$.
\end{lem}
\begin{proof}
The first inclusion is evident.
To prove the second inclusion, choose $z$ from $Y'$. Then, we can write $z=z_1+h_1+y_1$ for some $z_1\in A_j=Q_G^{-1}(\overline{B}_G(a_j,r))\cap V$, $h_1\in B_H(2r)$, and $y_1\in \overline{B}_G(R'+1)$.
There exists $y_2\in\overline{B}_G(r)$ such that $z_1-y_2\in Q_G^{-1}(a_j)$. 
Since $Q_G(P_G h_1)=0$, $|P_G h_1|_H<2r$, and $|Q_G h_1|_H<2r$, $z$ is decomposed as
\[
z=(z_1-y_2+P_G h_1)+(y_2+Q_G h_1+y_1),
\]
where 
\[
z_1-y_2+P_G h_1
\in Q_G^{-1}(a_j)\cap\bigl(V+\overline{B}_G(r)+B_H(2r)\bigr)\\
\subset Q_G^{-1}(a_j)\cap U
\quad\text{(from \Eq{V})}
\]
and
\[
y_2+Q_G h_1+y_1
\in \overline{B}_G(r)+B_G(2r)+\overline{B}_G(R'+1)
\subset B_G(R'+1+3r).
\]
Since $z\in U$, we conclude that $z\in X$.
\end{proof}
Let $\gm\in(0,1/2]$. We define a map $T_{\gm}\colon E\to E$ as
\begin{align*}
T_{\gm}(z)&:=P_G(z)+(1-\gm)Q_G(z)+\gm a_j\\
&=z+\gm(a_j-Q_G(z)).
\end{align*}
Then, for any $w\in E$, $T_{\gm}(w+G)=w+G$ and $T_{\gm}|_{w+G}$ is a homothety on $w+G$ that is centered at $P_G(w) + a_j$ with a magnification ratio $1-\gm$.

From a simple calculation, the induced measure of $\mu$ by the map $T_{\gm}$, denoted by $\mu\circ T_{\gm}^{-1}$, is absolutely continuous with respect to $\mu$, and the Radon--Nikodym derivative $d(\mu\circ T_{\gm}^{-1})/d\mu$ is uniformly bounded in $\gm$ on $Q_G^{-1}(C)$ for any compact set $C$ of $G$.

Let $X_\gm:=T_\gm^{-1}(X)$.
From the definitions, $X$ and $X_\gm$ are $H$-convex and belong to $\fM(E)$.
Therefore, $X$ and $X_\gm$ are moderate, and we can consider the function spaces $W^{1,2}(X)$ and $W^{1,2}(X_\gm)$.
We also note that $X\subset X_\b\subset X_\gm$ if $0<\b<\gm$. 
We define a function $g_\gm$ on $X_\gm$ by 
\[
g_\gm(z)=g(T_\gm(z))\quad\text{for } z\in X_\gm.
\]
Then, for a sufficiently small $\gm$, 
\begin{equation}\label{eq:ggm}
\{g_\gm\ne0\}\subset Y
\end{equation}
by \Lem{g} and \Eq{Y}.
Hereafter, we consider only such a small $\gm$, say, in the interval $(0,\gm_0]$ for some $\gm_0>0$.

The following lemma is intuitively evident; nonetheless, we have provided the proof.
\begin{lem}\label{lem:appendix}
Function $g_\gm$ belongs to $W^{1,2}(X_{\gm})$.
Moreover, $g_\gm|_X$ converges to $g|_X$ in $W^{1,2}(X)$ as $\gm\downarrow0$.
\end{lem}
\begin{proof}
First, we prove that $g_\gm\in W^{1,2}(X_{\gm})$ and 
\begin{equation}\label{eq:Dg}
Dg_\gm=(I-\gm Q_G)\bigl((Dg)\circ T_\gm\bigr),
\end{equation}
where $I$ denotes the identity operator on $H$.
Since 
\begin{equation}\label{eq:density}
\int_{X_\gm} g_\gm^2\,d\mu=\int_X g^2\,d(\mu\circ T_{\gm}^{-1})
\le\int_Y g^2\left\|\frac{d(\mu\circ T_{\gm}^{-1})}{d\mu}\right\|_{L^\infty(Y)}d\mu<\infty,
\end{equation}
we obtain $g_\gm\in L^2(X_\gm)$.
Similarly, we have $(I-\gm Q_G)\bigl((Dg)\circ T_\gm\bigr)\in L^2(X_\gm\to H)$.
Take any $h\in K\setminus\{0\}$ and define $k=(I-\gm Q_G)h\in K\setminus\{0\}$.
For $\mu_{k^\perp}$-a.e.\,$x\in k^\perp$, there exists an absolutely continuous version $\tilde g_k(x,\cdot)$ of $g_k(x,\cdot)$ such that
\[
  \la(Dg)(x+sk),k\ra=\frac{\partial\tilde g_k}{\partial s}(x,s)
  \quad\text{for a.e.\,}s\in I_{x,k}^X.
\]
For $x\in h^\perp$, 
\[
T_\gm(x+sh)=x+sh+\gm(a_j-Q_G(sh))=x+\gm a_j+sk
\]
and 
\[
\la x+\gm a_j,k\ra=\la x+\gm a_j,h-\gm Q_G h\ra
=-\gm\la Q_G x-(1-\gm)a_j,h\ra.
\]
Therefore, by letting 
\[
b={\gm\la Q_G x-(1-\gm)a_j,h\ra}/|k|_H^2
\quad\text{and}\quad
x'=x+\gm a_j+bk,
\]
we have $x'\in k^\perp$, $I_{x',k}^X+b=I_{x,h}^{X_\gm}$, and $\tilde g_k(x',\cdot-b)$ is an absolutely continuous version of $(g_\gm)_h(x,\cdot)$ on $I_{x,h}^{X_\gm}$.
Moreover,
\begin{align*}
\frac{\partial}{\partial s}\tilde g_k(x',s-b)
&=\la (Dg)(x'+(s-b)k),k\ra\\
&=\la(Dg)(T_\gm(x+sh)),(I-\gm Q_G)h\ra\\
&=\left\la (I-\gm Q_G)\bigl((Dg)(T_\gm(x+sh))\bigr),h\right\ra.
\end{align*}
This implies that $g_\gm\in W^{1,2}(X_\gm)$ and $Dg_\gm=(I-\gm Q_G)\bigl((Dg)\circ T_\gm\bigr)$.

Next, we prove that $g_\gm|_X$ converges to $g|_X$ in $W^{1,2}(X)$ as $\gm\downarrow0$.
For $\mu_{G^\perp}$-a.e.\,$x\in G^\perp$, the convergence of $g_\gm|_{(x+G)\cap X}$ to $g|_{(x+G)\cap X}$ in $L^2((x+G)\cap X,\mu_{x+G}|_{(x+G)\cap X})$ is proved in a standard way as follows. 
 For $x\in G^\perp$, define 
 \[
 g^*(z)=\begin{cases}g(z)&\text{if $z\in (x+G)\cap X$,}\\0&\text{if $z\in (x+G)\setminus X$.}\end{cases}
 \]
For $\mu_{G^\perp}$-a.e.\,$x$, $g^*$ belongs to $L^2(x+G,\mu_{x+G})$. 
 Let $\Psi$ be a smooth function on $G$ with compact support and $\int_{G}\Psi(y)\,\lm_m(dy)=1$, where $m=\dim G$ and $\lm_m$ denotes the Lebesgue measure on $G$.
 For each $\eps>0$, define a smooth function $\psi_\eps$ on $x+G$ by $\psi_\eps(z)=\eps^{-m}\int_{G}g^*(z-y)\Psi(\eps^{-1}y)\,dy$.
 Then, denoting $L^2((x+G)\cap X,\mu_{x+G}|_{(x+G)\cap X})$ by $L^2((x+G)\cap X)$, we have
\begin{align}
 &\|g_\gm|_{(x+G)\cap X}-g|_{(x+G)\cap X}\|_{L^2((x+G)\cap X)}\nonumber\\
 &\le \|g_\gm|_{(x+G)\cap X}-(\psi_\eps\circ T_\gm)|_{(x+G)\cap X}\|_{L^2((x+G)\cap X)}\nonumber\\
&\phantom{{}\le}+\|(\psi_\eps\circ T_\gm)|_{(x+G)\cap X}-\psi_\eps|_{(x+G)\cap X}\|_{L^2((x+G)\cap X)}\nonumber\\
&\phantom{{}\le}+\|\psi_\eps|_{(x+G)\cap X}-g|_{(x+G)\cap X}\|_{L^2((x+G)\cap X)}.
\label{eq:3e}
\end{align}
The last term of \Eq{3e} converges to $0$ as $\eps\downarrow0$, as does the first term of \Eq{3e} by using an estimate similar to \Eq{density}.
From the dominated convergence theorem, the second term converges to $0$ as $\gm\downarrow0$.
Therefore, by letting $\gm\downarrow0$ and $\eps\downarrow0$,
$g_\gm|_{(x+G)\cap X}$ converges to $g|_{(x+G)\cap X}$ in $L^2((x+G)\cap X,\mu_{x+G}|_{(x+G)\cap X})$.
By integrating $\|g_\gm|_{(x+G)\cap X}-g|_{(x+G)\cap X}\|_{L^2((x+G)\cap X)}^2$ over $G^\perp$ with respect to $\mu_{G^\perp}(dx)$ and by using the dominated convergence theorem, we obtain  $\|g_\gm|_X-g|_X\|_{L^2(X)}^2\to 0$ as $\gm\downarrow0$.

Similarly, we can show that 
\begin{equation}\label{eq:Dggm}
\bigl((Dg)\circ T_\gm\bigr)|_X\to (Dg)|_X\quad \text{in }L^2(X\to H)\text{ as }\gm\downarrow0.
\end{equation}
From \Eq{Dg}, $\{Dg_\gm|_X\}_{\gm\in(0,\gm_0]}$ is bounded in $L^2(X\to H)$.
Therefore, $\{g_\gm|_X\}_{\gm\in(0,\gm_0]}$ is bounded in $W^{1,2}(X)$, and it is weakly relatively compact.
Since any accumulation point should be $g|_X$, $g_\gm|_X$ converges weakly to $g|_X$ in $W^{1,2}(X)$.
Since $\lim_{\gm\downarrow0}\|g_\gm|_X\|_{W^{1,2}(X)}=\|g|_X\|_{W^{1,2}(X)}$ in view of \Eq{Dg} and \Eq{Dggm}, we conclude that $g_\gm|_X$ converges to $g|_X$ in $W^{1,2}(X)$ as $\gm\downarrow0$.
\end{proof}
We extend the defining set of $g_\gm$ to $X_\gm\cup U$ by letting $g_\gm(z)=0$ for $z\in U\setminus X_\gm$.
Since
\[
\{g_\gm\ne0\}\cap U\subset Y\cap U\subset(Y+B_H(r))\cap U=Y'\subset X\subset X_\gm,
\]
we have $g_\gm|_U\in W^{1,2}(U)$.
\begin{lem}\label{lem:key}
  It holds that $\overline{Y\cap U}^H\subset X_{\gm}$.
\end{lem}
\begin{proof}
Let $z\in \overline{Y\cap U}^H$.
Then, $z$ is described as
\[
  z=z_1+h_1+y_1+h_2
\]
for some $z_1\in Q_G^{-1}(\overline{B}_G(a_j,r))\cap V$, $h_1\in B_H(r)$, $y_1\in\overline{B}_G(R'+1)$, and $h_2\in B_H(r)$.
Then, $z\in V+ B_H(2r)+G\subset U+G$.
Therefore, $U\cap (z+G)\ne\emptyset$.

From \Lem{section}, we have ${\overline{U}^H\cap(z+G)}=\overline{U\cap(z+G)}^{z+G}$.
Since $z\in{\overline{U}^H\cap(z+G)}$, we have $z\in\overline{U\cap(z+G)}^{z+G}$.
Moreover, there exists $y_2\in\overline{B}_G(r)$ such that $z_1-y_2\in Q_G^{-1}(a_j)$.
Then,
\begin{align*}
z&=(z_1-y_2+P_G(h_1+h_2))+(Q_G(h_1+h_2)+y_1+y_2)\\
&\in \left(Q_G^{-1}(a_j)\cap\left(V+\overline{B}_G(r)+B_H(2r)\right)\right)+\left(B_G(2r)+\overline{B}_G(R'+1)+\overline{B}_G(r)\right)\\
&\subset \left(Q_G^{-1}(a_j)\cap U\right)+G\quad\text{(from \Eq{V})}.
\end{align*}
Therefore, $\bigl(U\cap(z+G)\bigr)\cap Q_G^{-1}(a_j)\ne\emptyset$.
Combining this with the facts that $U\cap(z+G)$ is convex in $z+G$ and $z\in\overline{U\cap(z+G)}^{z+G}$, we obtain $T_{\dl}(z)\in U\cap(z+G)\subset U$ for all $\dl\in(0,\gm]$.
Furthermore, if $\dl$ is sufficiently small, we have $|T_{\dl}(z)-z|_H<r$, which implies that $T_{\dl}(z)\in Y+B_H(r)$.
Therefore, $T_{\dl}(z)\in Y'$ for such $\dl$.
This implies that $z\in  T_{\dl}^{-1}(Y')\subset X_\dl\subset X_{\gm}$.
\end{proof}
From this lemma, $\overline{Y\cap U}^H\cap (E\setminus X_{\gm})=\emptyset$. 
Since both $Y\cap U$ and $X_{\gm}$ are $H$-open and belong to $\fM(E)$, both $\overline{Y\cap U}^H\,(=E\setminus(Y\cap U)^{H\text{-ext}})$ and $E\setminus X_{\gm}$ are quasi-closed from \Prop{m}.

Let $m\in\N$. 
We can take open subsets $O_{m,1}$ and $O_{m,2}$ of $E$ such that $\Cp_{1,2}(O_{m,i})<1/m$ $(i=1,2)$ and both $\overline{Y\cap U}^H\setminus O_{m,1}$ and $(E\setminus X_{\gm})\setminus O_{m,2}$ are closed in $E$.
Since $\Cp_{1,2}$ is tight (see, e.g., \cite{RS92}), there exists an open set $O_{m,3}$ such that $\Cp_{1,2}(O_{m,3})<1/m$ and $E\setminus O_{m,3}$ is compact.
Let $O_m=O_{m,1}\cup O_{m,2}\cup O_{m,3}$.
We define $C_{m}=\overline{Y\cap U}^H\setminus O_{m}$ and $C'_{\gm,m}=(E\setminus X_\gm)\setminus O_m$.
Since both sets are compact and $C_{m}\cap C'_{\gm,m}=\emptyset$, $\sd_E(C_{m},C'_{\gm,m})=:\a>0$.
Define $C''_{\gm,m}=C'_{\gm,m}+\{z\in E\mid |z|_E\le \a/2\}$ and
\[
\rho_{\gm,m}(z)=\frac{\sd_E(z,C''_{\gm,m})}{\sd_E(z,C_{m})+\sd_E(z,C''_{\gm,m})},\quad
z\in E.
\]
Then, $\rho_{\gm,m}$ is $H$-Lipschitz, $0\le\rho_{\gm,m}\le1$ on $E$, $\rho_{\gm,m}=1$ on $C_{m}$, and $\rho_{\gm,m}=0$ on $C''_{\gm,m}$.

We note that $g_\gm|_{X_\gm}\in W^{1,2}(X_\gm)\cap L^\infty(X_\gm)$ and $1-e_{O_m}=0$ on $O_m$, where $e_{O_m}$ is the equilibrium potential of $O_m$.
Moreover, $X_\gm$, $O_m$, and $C'_{\gm,m}+\{z\in E\mid |z|_E<\a/2\}$ are all $H$-open sets, and their union is equal to $E$.
Therefore, $g_{\gm,m}:=g_\gm\cdot(1-e_{O_m})\cdot\rho_{\gm,m}$ is well-defined as an element of $W^{1,2}(E)\cap L^\infty(E)$ and $g_{\gm,m}=g_\gm\cdot(1-e_{O_m})$ on $U$ from \Eq{ggm}.
Then, 
\[
\|g-g_{\gm,m}|_U\|_{W^{1,2}(U)}\le\|g-g_\gm|_U\|_{W^{1,2}(U)}+\|g_\gm|_U-g_{\gm,m}|_U\|_{W^{1,2}(U)}
\]
and
\begin{align}
&\|g_\gm|_U-g_{\gm,m}|_U\|_{W^{1,2}(U)}^2 \nonumber\\
&=\|g_\gm|_U-g_\gm(1-e_{O_m})|_U\|_{W^{1,2}(U)}^2\nonumber\\
&=\|(g_\gm e_{O_m})|_U\|_{W^{1,2}(U)}^2\nonumber\\
&\le 2\|((Dg_\gm)e_{O_m})|_U\|_{L^2(U)}^2+2M^2\|(De_{O_m})|_U\|_{L^2(U)}^2
+M^2\|e_{O_m}|_U\|_{L^2(U)}^2 \label{eq:conv}\\
&\to0 \quad\text{as }m\to\infty.\nonumber
\end{align}
Here, we note that the first term of \Eq{conv} converges to $0$ since $|(Dg_\gm)e_{O_m}|_H\le |Dg_\gm|_H$ and $\bigl(|(Dg_\gm)e_{O_m}|_H\bigr)|_U$ converges to $0$ in measure $\mu|_U$.

By combing these estimates with \Lem{appendix}, $\lim_{\gm\downarrow0}\lim_{m\to\infty}\|g-g_{\gm,m}|_U\|_{W^{1,2}(U)}=0$.
That is, $g$ can be approximated in $W^{1,2}(U)$ by elements of $W^{1,2}(E)|_U$.
This completes the proof of \Thm{main}.
\section{Concluding remarks}
Let $U$ be the same as in \Thm{main}.
\begin{enumerate}
\item Feyel and \"Ust\"unel~\cite{FU00} proved the following logarithmic Sobolev inequality on $U$.
\begin{thm}[cf.\ {\cite[Theorem~6.4]{FU00}}]\label{th:LS}
For any $f\in\FCb^1(E)$,
\begin{equation}\label{eq:logsob}
  \mint_U f^2\log\left(f^2\left/\mint_U f^2\,d\mu\right.\right)d\mu \le 2 \mint_U |Df|_H^2\,d\mu.
\end{equation}
Here, $\mint_U \cdots\, d\mu:=\mu(U)^{-1}\int_U \cdots\, d\mu$ denotes
the normalized integral on $U$.
\end{thm}
(A more general result is proved in \cite[Theorem~6.4]{FU00}.)
\Thm{main} implies that \eqref{eq:logsob} is also true for all $f\in W^{1,2}(U)$ from the approximation procedure. 
\item We consider a Markov process associated with $(\cE^U,W^{1,2}(U))$. From \cite[Theorem~2.1]{Fu00} (see also \cite{RS92}), the closure of $(\cE^U,\cF C_b^1(E)|_U)$ is a quasi-regular local Dirichlet form on $L^2(\bar U,\mu|_U)$, where $\bar U$ denotes the closure of $U$ in $E$. Therefore, there is an associated diffusion process $\{X_t\}$ on $\bar U$. 
Moreover, suppose that the indicator function of $U$ belongs to $BV(E)$, which is defined in \cite{Fu00,FH01}. Then, we have the Skorohod-type representation of $\{X_t\}$ (\cite[Theorem~4.2]{FH01}):
\[
  X_t(\omega)-X_0(\omega)
  =W_t(\omega)-\frac12\int_0^t X_s(\omega)\,ds
  +\frac12\int_0^t \sigma_U(X_s(\omega))\,dA_s^{\|D 1_U\|}(\omega),
  \quad t\ge0,
\]
where $\{W_t\}$ is the $E$-valued Brownian motion starting at $0$, $\sigma_U$ is an $H$-valued function on $E$, and $A^{\|D 1_U\|}$ is a positive continuous additive functional. $\sigma_U$ and $A^{\|D 1_U\|}$ formally correspond to the vector field normal to the boundary $U^\partial$ of $U$ and the additive functional induced by the surface measure of $U^\partial$, respectively (see \cite{FH01} for more precise descriptions).
In \cite{Fu00,FH01}, $\{X_t\}$ is called the {\em modified} reflecting Ornstein--Uhlenbeck process (for more general $U$) because the domain of the Dirichlet form is defined by the smallest extension of $\FCb^1(E)|_U$, and in general, it is not clear whether it really is a natural one. Under the assumptions of \Thm{main}, the domain is equal to $W^{1,2}(U)$, and $\{X_t\}$ can be aptly called the {\em true} reflecting Ornstein--Uhlenbeck process on $\bar U$. 
\end{enumerate}
\bigskip

\noindent
\textit{Acknowledgment.} A part of this article was written during a visit to the Scuola Normale Superiore in Pisa.  I would like to thank Professor Luigi~Ambrosio for his hospitality and insightful discussions. I would also like to thank Professors Kazuhiro~Kuwae and Michael~R\"ockner for their useful comments.



%% References


\bibliographystyle{model1-num-names}

\begin{thebibliography}{00}


\bibitem{AKR90} S. Albeverio, S. Kusuoka, and M. R\"ockner, On partial integration in infinite-dimensional space and applications to Dirichlet forms, J. London Math. Soc. (2) \textbf{42} (1990), 122--136. 
\bibitem{AR90} S. Albeverio and M. R\"ockner, Classical Dirichlet forms on topological vector spaces---closability and a Cameron--Martin formula, J. Funct. Anal. \textbf{88} (1990), 395--436.
\bibitem{Bo} V. I. Bogachev, \textit{Gaussian measures}, Mathematical Surveys and Monographs \textbf{62}, Amer. Math. Soc., Providence, RI, 1998.
\bibitem{BH} N. Bouleau and F. Hirsch, \textit{Dirichlet forms and analysis on Wiener space}, de Gruyter Studies in Mathematics \textbf{14}, Walter de Gruyter, Berlin, 1991.
\bibitem{Bou} N. Bourbaki, \textit{General topology: Chapters 5--10},
Translated from the French, Reprint of the 1966 edition, Elements of Mathematics (Berlin), Springer, Berlin, 1989.
\bibitem{DM} C. Dellacherie and P.-A. Meyer, \textit{Probabilities and Potential}, North-Holland, 1978.
\bibitem {Eb} A. Eberle, \textit{Uniqueness and non-uniqueness of semigroups generated by singular diffusion operators}, Lecture Notes in Math. \textbf{1718}, Springer-Verlag, Berlin, 1999.
\bibitem{ES93} O. Enchev and D. W. Stroock, Rademacher's theorem for Wiener functionals, Ann. Probab. {\bf 21} (1993), 25--33.
\bibitem{FU00} D. Feyel and A. S. \"Ust\"unel, The notion of convexity and concavity on Wiener space, J. Funct. Anal. \textbf{176} (2000), 400--428. 
\bibitem{Fu00} M. Fukushima, $BV$ functions and distorted Ornstein Uhlenbeck processes over the abstract Wiener space, J. Funct. Anal. {\bf 174} (2000), 227--249.
\bibitem{FH01} M. Fukushima and M. Hino, On the space of BV functions and a related stochastic calculus in infinite dimensions, J. Funct. Anal. \textbf{183} (2001), 245--268.
\bibitem{FOT} M. Fukushima, Y. Oshima and M. Takeda, \textit{Dirichlet forms and symmetric Markov processes}, de Gruyter Studies in Mathematics {\bf 19}, Walter de Gruyter, Berlin, 1994.
\bibitem{Hi03} M. Hino, On Dirichlet spaces over convex sets in infinite dimensions, in \textit{Finite and infinite dimensional analysis in honor of Leonard Gross (New Orleans, LA, 2001)}, 143--156, Contemp. Math. \textbf{317}, Amer. Math. Soc., 2003.
\bibitem{Hi10} M. Hino, Sets of finite perimeter and the Hausdorff--Gauss measure on the Wiener space, J. Funct. Anal. \textbf{258} (2010), 1656--1681.
\bibitem{Ku82} S. Kusuoka, The nonlinear transformation of Gaussian measure on Banach space and absolute continuity (I), J. Fac. Sci. Univ. Tokyo Sect. IA Math. \textbf{29} (1982), 567--597.
\bibitem{MR} Z. M. Ma and M. R\"ockner, \textit{Introduction to the theory of (nonsymmetric) Dirichlet forms}, Universitext, Springer, Berlin, 1992.
\bibitem{Ma} P. Malliavin, \textit{Stochastic analysis}, Grundlehren der Mathematischen Wissenschaften \textbf{313}, Springer-Verlag, Berlin, 1997.
\bibitem{RR05} J. Ren and M. R\"ockner, A remark on sets in infinite dimensional spaces with full or zero capacity, in \textit{Stochastic analysis: classical and quantum}, 177--186, World Sci. Publ., Hackensack, NJ, 2005.
\bibitem{Ro} R. T. Rockafellar, \textit{Convex analysis},  Princeton Mathematical Series \textbf{28}, Princeton University Press, Princeton, 1970.
\bibitem{RSc99} M. R\"ockner and A. Schied, Rademacher's theorem on configuration spaces and applications, J. Funct. Anal. \textbf{169} (1999), 325--356.
\bibitem{RS92} M. R\"ockner and B. Schmuland, Tightness of general $C\sb{1,p}$ capacities on Banach space, J. Funct. Anal. {\bf108} (1992), 1--12.
\bibitem{Us} A. S. \"Ust\"unel, \textit{Analysis on Wiener space and applications}, preprint. \texttt{arXiv:1003.1649}
\end{thebibliography}

\begin{thebibliography}{Hi11}
\bibitem[Hi11]{H11} M. Hino, Dirichlet spaces on $H$-convex sets in Wiener space, Bull.\ Sci.\ Math.\ \textbf{135} (2011), 667--683.
\end{thebibliography}



%\newpage
\bigskip\ \\ \bigskip
\newcommand{\fO}{\mathfrak{O}}
\begin{center}
{\Large Erratum to
``Dirichlet spaces on $H$-convex sets in Wiener space''
[Bull.\ Sci.\ Math.\ 135 (2011) 667--683]}
\end{center}
In Definition~2.5 of \cite{H11}, the statement $A^{H\text{-int}}=(A^{H\text{-ext}})^{H\text{-ext}}$ is false. A counterexample is $A=B_H(0,1)\setminus\{0\}$.
This error affects Proposition~2.10 (ii) (iii) and Proposition~2.11.
We modify their statements and proofs as follows.
First, although the original proof of Proposition~2.11 is invalid since it uses Proposition~2.10, the claim is true by the following proof.
\begin{proof}[Proof of Proposition~2.11.]
We may assume that $f$ is a real-valued function.
Take a countable dense subset $H_0$ of $H$.
For $n\in\N$, define
\[
  f_n(z)=\inf_{h\in H_0}\{f(z+h)+n|h|_H\},\quad z\in E.
\]
Then, $f_n$ is an $\fM(E)$-measurable function. From the $H$-continuity of $f$, we can replace $H_0$ by $H$ in the above definition.
Therefore, $f_n$ is $H$-Lipschitz.
Moreover, $f_n(z)$ is nondecreasing in $n$ and converges to $f(z)$ for all $z\in E$.
In particular, $\{f>a\}=\bigcup_{n\in\N}\{f_n>a\}$ for any $a\in\R$.
From Lemma~2.9, this set is quasi-open.
By replacing $f$ with $-f$, we also obtain that $\{f<a\}$ is quasi-open for any $a\in\R$.
This completes the proof.
\end{proof}

We set 
\[
\fO(E)=\left\{A\subset E\;\vrule\; \parbox{0.72\hsize}{there exists an $\fM(E)$-measurable and $H$-continuous function $f$ on $E$ such that $A=\{f>0\}$ up to $\Cp_{1,2}$-null set}\right\}.
\]
Then, we modify the statements of Proposition~2.10 (ii) and (iii) as follows, the proof of which is straightforward from Proposition~2.11.
\bigskip

\noindent\textbf{Proposition~2.10.} (corrected)
\begin{enumerate}
\item[(ii)] If $A\in\fO(E)$, then $A$ is quasi-open.
\item[(iii)] If $E\setminus A\in\fO(E)$, then $A$ is quasi-closed.
\end{enumerate}
\medskip

The author does not determine whether all $H$-open sets $A$ in $\fM(E)$ belong to $\fO(E)$ or not.
Although any $H$-open set $A$ is described as $A=\{\sd_H(\cdot,E\setminus A)>0\}$, it is not clear whether $\sd_H(\cdot,E\setminus A)$ is an $\fM(E)$-measurable function in general.
The following are some criteria for $\fO(E)$.
\bigskip

\noindent\textbf{Proposition~A.}
Let $A$ be an $H$-open set in $\fM(E)$.
If either of the following is satisfied in addition, then $A\in\fO(E)$.
\begin{enumerate}
\item $E\setminus A$ is a Suslin set.
\item $A=(A^{H\text{-ext}})^{H\text{-ext}}$.
\end{enumerate}
\begin{proof}
If (i) holds, the function $f(z):=\sd_H(z,E\setminus A)\wedge1$ is shown to be universally measurable as in the proof of Lemma~2.6.
Since $f$ is $H$-Lipschitz and $A=\{f>0\}$, we conclude that $A\in\fO(E)$.
If (ii) holds, $A=\{\sd_H(\cdot,A^{H\text{-ext}})>0\}$.
Since $\sd_H(\cdot,A^{H\text{-ext}})\wedge1$ is $\fM(E)$-measurable and $H$-Lipschitz as in the proof of Proposition~2.10 (i), $A$ belongs to $\fO(E)$.
\end{proof}
Note that if $A$ is an $H$-open and $H$-convex set, then condition~(ii) in the above proposition holds.
Therefore, the proof of Lemma~3.2 is valid by taking the set $O$ in the proof so that $E\setminus O$ is convex in addition.
All other arguments in Section~3 work without modification.
In particular, Theorem~1.2, which is the main theorem of \cite{H11}, remains true.
\bibliographystyle{model1-num-names}

\end{document}